\documentclass[11pt]{amsart}
\usepackage{bm,amssymb,amsmath,amsthm,
comment, mathtools,color,stmaryrd, latexsym,graphicx,bm}
\usepackage{amsrefs}

\usepackage{pgf,tikz,tikz-cd}
\usetikzlibrary{matrix,arrows,positioning, decorations.pathmorphing,shapes.geometric}
\usepackage{amssymb, 
enumerate}

\newtheorem{thm}{Theorem}[section]
\newtheorem{lem}[thm]{Lemma}

\newtheorem{prop}[thm]{Proposition}

\theoremstyle{definition}

\newtheorem{rems}[thm]{Remarks}  

\newtheorem{ex}[thm]{Example}

\usepackage{amsrefs} 

\newenvironment{newlist}
   {\begin{list}{}{\setlength{\labelsep}{0.25cm}
\setlength{\itemsep}{0.3cm}
\setlength{\topsep}{0.2cm}
                   \setlength{\labelwidth}{0.65cm}
                      \setlength{\leftmargin}{0.9cm}}}
   {\end{list}}
\usepackage{tikz-cd}

\usepackage{mathrsfs} 
\usepackage{bigstrut}
\usepackage{tikz}
\usetikzlibrary{patterns}
\usetikzlibrary{matrix,arrows,decorations,shapes}

\pgfmathsetmacro{\nodebasesize}{1} 
\pgfmathsetmacro{\nodeinnersep}{0.1}

\newcommand{\twiddle}[1]{\smash{\underset{\raise.375ex\hbox{$\smash\sim$}}
       {#1}}\vphantom{\underline{#1}}} 

\newcommand{\newzero}[1]{\smash{\overset{\lower 1ex\hbox{$\smash \tiny{=}$}}
       {#1}}\vphantom{\overline{#1}}}

\newcommand{\newone}[1]{\smash{\overset{\lower.3ex\hbox{$\smash \tiny{=}$}}
       {#1}}\vphantom{\overline{#1}}}

\newcommand{\CM}{\mathcal{M}}

\newcommand{\fntU}{\fnt{U}}

\newcommand{\du}{\smash{\,\cup\kern-0.45em\raisebox{1ex}{$\cdot$}}\,\,}

\newcommand{\pc}{\preccurlyeq}

\newenvironment{widenewlist}
   {\begin{list}{}{\setlength{\labelsep}{0.2cm}
                   \setlength{\labelwidth}{1.6cm}
                      \setlength{\leftmargin}{1.7cm}}}
   {\end{list}}

\newcommand{\pcq}{\sqsubseteq}   

\newcommand{\class}[1]{\mathcal{#1}}
\newcommand{\cat}[1]{\boldsymbol{\mathscr{#1}}}  
\newcommand{\alg}[1]{\mathbf{#1}}
\newcommand{\fnt}[1]{\mathsf{#1}}
\newcommand{\E}{\fnt{E}}
\newcommand{\D}{\fnt{D}}
\newcommand{\ope}[1]{\mathbb{#1}}
\newcommand{\defn}[1]{\emph{#1}}

\newcommand{\CCD}{\class{D}}
\newcommand{\CP}{\class{P}}

\newcommand{\fntH}{\fnt{H}}
\newcommand{\fntK}{\fnt{K}}

\newcommand{\dotbigcup}{\overset{\raise.6ex\hbox{$\smash\cdot$}\,}{\smash\bigcup\,}} 

\newcommand{\dotbigcupdisp}{\mbox{\Large{$\textstyle\overset{\raise.65ex\hbox{$\smash\cdot$}}{\smash\bigcup}$}}}

\newcommand{\A}{\alg{A}}
\newcommand{\B}{\alg{B}}
\newcommand{\C}{\alg{C}}
\newcommand{\M}{\alg{M}}
\newcommand{\X}{\alg{X}}
\newcommand{\Y}{\alg{Y}}
\newcommand{\Z}{\alg{Z}}
\renewcommand{\P}{\alg{P}}

\newcommand{\Tp}{\mathcal{T}}

\newcommand{\twoT}{\twiddle{\two}}
\newcommand{\Zed}{\alg{Z}}
\newcommand{\two}{\boldsymbol 2}
\newcommand{\SA}{\cat{S\!A}}

\newcommand{\CA}{\cat{A}}

\newcommand{\Q}{\alg{Q}}

\newcommand{\CX}{\mathcal{X}}

\newcommand{\w}{\omega}
\renewcommand{\epsilon}{\varepsilon}

\newcommand{\CMT}{\twiddle{\CM}}

\newcommand{\stwiddle}[1]{\smash{\underset{\smash{\raise.1ex\hbox{\small$\sim$}}}
                         {\mathbf{#1}}}\vphantom{#1}}
 \font\lmi=cmmi5
 \def\ltail{\raise.25ex\hbox{\lmi\char'076 }\!}
 \def\rtail{\raise.25ex\hbox{\lmi\char'074 }\!}
 
\newcommand{\twiddleeven}{
\smash{\underset{\lower.65ex\hbox{$\smash{
\widetilde{\phantom{rr\!\!r}}}
$}
}{
\C_n^{\sigma^\flat}
}}}

\newcommand{\twiddleodd}{
\smash{\underset{\lower.5ex\hbox{$
\smash{\widetilde{\phantom{mm}}}
$}
}{     \C_{n+1}^\sigma
}}}
\newcommand{\twiddleODD}{
\smash{\underset{\lower.5ex\hbox{$\smash{\widetilde{\phantom{mm}}}
$}
}{      \C_{n+1}^{\text{\tiny{End}}}
}}}

\newcommand{\sub}[1]{_{_{\kern-.9pt{\scriptstyle #1}}}}
\newcommand{\medsub}[2]{#1\lower0.6ex\hbox{$\scriptstyle{#2}$}}
\newcommand{\eA}[1]{\medsub e {\kern-0.75pt\A\kern-0.75pt}(#1)}
\newcommand{\esub}[1]{\medsub e {\kern-0.75pt #1 \kern-0.75pt}}
\newcommand{\epsub}[1]{\medsub \varepsilon {\kern-1.25pt #1}}
\newcommand{\esubA}{\medsub e {\kern-0.75pt\A\kern-0.75pt}}

\newcommand{\ksubB}{\medsub k {\kern-0.75pt\B\kern-0.75pt}}

\DeclareMathOperator{\ISP}{\ope{ISP}}
\DeclareMathOperator{\IScPn}{\ope{IS}_c\ope{P}}

 \DeclareMathOperator{\HSP}{\ope{HSP}}

\newcommand{\id}{\mathop{}\mathopen{}\mathrm{id}} 
\DeclareMathOperator{\img}{im}

\newcommand{\three}{\boldsymbol 3}

\newcommand{\zero}{\overline{\boldsymbol 0}}

\newcommand{\one}{\overline{\boldsymbol 1}}
\newcommand{\Klat}{\class{K}lat}
\newcommand{\Kalg}{\class{K}alg}

\begin{document}

\title
{Piggybacking over 
unbounded distributive lattices}

%
%
%

\author{Leonardo M. Cabrer}
\email{lmcabrer@yahoo.com.ar}
\address{London, UK}

\author{Hilary A. Priestley}
\email{hap@maths.ox.ac.uk}

\address{Mathematical Institute \\
University of Oxford, UK}

\begin{abstract}
This paper fills a gap in the literature on natural duality
theory. 
It concerns dual representations of categories of  
distributive-lattice-based algebras  in which the lattice reducts are not assumed to have bounds.
  The development of theory
to parallel what is known for the exhaustively-studied bounded case was
initially driven by need.    This arose in connection with a major
investigation of
Sugihara algebras and Sugihara monoids.  
The theorems in this 
paper  apply 
 in a systematic way to a range of examples:
varieties of Sugihara 
 type;  other classes of algebras 
previously treated ad hoc; and 
  further classes as required.  
\end{abstract}
\keywords{Distributive lattice, Priestley duality, natural duality,
 multisorted duality, piggyback method, Sugihara 
algebra}
\subjclass[2010]{Primary: 08C20;
Secondary:  03G25, 06D50}
\maketitle


\section{Introduction}\label{sec:intro}

There is a gap in the literature on duality theory for
distributive-lattice-based algebras.  The present  paper  addresses this.   By doing so, it paves the way to
 applications to a range of classes which   provide algebraic semantics
for certain well-studied propositional logics.    
The development of the theory we present was prompted by a study by the authors
of Sugihara algebras  and Sugihara monoids \cites{CPsug,CPmult,CPfree}.  These varieties have attracted 
much interest, not least because they provide algebraic
semantics for classes of relevant logics with the mingle axiom;  see for example \cites{Du70,FG19,GM16,OR07}.
Algebras  of Sugihara type have
reducts in the variety $\CCD_u$
of unbounded distributive lattices.  That is, 
bounds~$0$ and~$1$ are not included as constants
in the language.

Duality  methods are  very well established as a tool for studying varieties and quasivarieties of algebras whose members have 
reducts in the variety $\CCD$ of  \emph{bounded} 
distributive lattices, where constants $0$ and~$1$ 
are included.  
Underpinning this successful endeavour have been 
two important dualities:  
Stone duality for Boolean algebras and Priestley duality for  $\CCD$.
These dualities exhibit extremely good behaviour: in the parlance of the now-classic text by Clark and Davey \cite{CD98}*{Chapters 3, 4, 6} they are both `strong' and `good'.   This makes them powerful
tools for translating algebraic problems into more amenable dual formulations and for gaining benefits in terms of computational complexity.  
Extensive catalogues of examples exist of the successful employment of duality methods for classes of algebras with Boolean or $\CCD$ reducts for which well-behaved dualities have  been devised. 
 Hitherto,  the
methodology has not been fully extended 
to cover  classes  with reducts in~$\CCD_u$.    

To provide context
we summarise the approaches adopted for 
classes of
$\CCD$-based
 algebras.  
For a given class
two alternatives 
present themselves:

\begin{enumerate}[(I)]  
\item
setting up a restricted Priestley duality  in which the additional operations are captured, usually by operations or relations, on the dual side;
\item 
 setting up a natural duality which emulates, as far as possible, the desirable
features possessed by the Stone and Priestley  
dualities.
\end{enumerate}
A recent paper \cite{DG17}  formalises~(I) in categorical terms.  For~(II), the key result 
guaranteeing feasibility is the NU Strong Duality
Theorem \cite{CD98}*{7.1.2} (and this does not
require the lattice reducts to be bounded). 
In general, 
 compromises will be unavoidable.  
Under~(I), coproducts, and in particular free algebras,
seldom have simple dual descriptions (whereas the
Stone and Priestley functors convert coproducts into
concrete (that is, cartesian) products).  On the plus side, under~(I),  algebras can be concretely represented in terms of families of sets.

How can (I) and (II) be reconciled so as to capitalise
on the  merits of each?  We highlight three influential 
developments.
The first was  Davey \& Werner's 
(simple) piggyback method which,  for a  quasivariety
$\CA$
with a  forgetful functor~$\fntU$
into a base quasivariety 
$\cat B$ for which an amenable natural duality was to hand,
guided the choice of a 
dualising
 alter ego for~$\CA$.
The method was however limited in scope.  
The next  major advance
 was the introduction by Davey \& Priestley~\cite{DP87} of multisorted natural dualities
and the extension to this setting of piggybacking.  
(The idea of employing  categories of multisorted structures is applicable
well beyond traditional duality theory and shows 
promise for the future---but this does not concern 
us here.)  
Multisorted piggybacking has proved very 
useful when the base variety is~$\CCD$; see for 
example \cites{DP87, CCP, CPcop, DBlat, HPLuk}. 
The third advance to be highlighted relates to that setting specifically, so $\cat \B= \CCD$.  
The present authors  \cite{CPcop}
   showed 
how a multisorted 
piggyback 
duality for a $\CCD$-based
quasivariety $
\CA$
connects in a transparent way to  Priestley 
duality as this applies to the $\CCD$-reducts,
and, under appropriate conditions, to a restricted
Priestley duality for $\CA$ itself.  

Above, 
the restriction to the base variety being 
\emph{bounded} distributive lattices is noteworthy.
Only isolated examples have been considered in the
unbounded case, and general theory has not been available.  We remedy this omission.

Our main results
  are
Theorems~\ref{thm:multpig} and~\ref{thm:RevEng}.
The first of these is our Multisorted Piggyback Duality
Theorem for classes~$\CA$ of $\CCD_u$-based algebras.
 The second relates the natural duality  in Theorem~\ref{thm:multpig} to Priestley duality for 
the $\CCD_u$-reduct of~$\CA$.
Together, these theorems do the same job  in reconciling (I) and~(II)
as  earlier papers do for the bounded case.  The 
theorems  we present find their first applications in \cites{CPmult,CPfree}.  The latter paper, devoted to
free algebras in varieties of Sugihara type, leads to
descriptions of the
underlying lattice structure of such algebras.  This
provides
information which  would  be challenging  
to obtain using  either method~(I) or method~(II) alone.

Theorems~\ref{thm:multpig} and~\ref{thm:RevEng}
 work smoothly.  But
we warn that  en route some 
 subtleties emerge which need
handling with due care; see Section~\ref{sec:JS}.
We believe that in writing this paper 
we do a service to those who potentially have 
uses for duality results in the unbounded case
but who would  baulk  at working out the technicalities  themselves.  
We shall assume some familiarity with duality theory,
as exemplified by Priestley duality for $\CCD$.  
The classic text by Clark \& Davey \cite{CD98} is used as a primary reference for natural duality theory
(our notation does not always align with that in \cite{CD98} but  is internally consistent).
Alternative sources are available:  the authors' cited papers, and Sections~\ref{sec:DL} and~\ref{sec:pigprelim} below, 
outline the material we need.  
We are able to coordinate the  proofs  
of Theorems~\ref{thm:multpig} and~\ref{thm:RevEng}.
 Note that
 the corresponding theorems in the
bounded case were first published in 1987 
\cite{DP87} and 2014 \cite{CPcop}.

In  Section~\ref{sec:exs} we discuss the application of 
our theorems to
\begin{itemize}
\item  {\bf Kleene lattices}

\item  finitely generated varieties of {\bf Sugihara algebras} and {\bf Sugihara monoids}, both odd and even cases,

\item  {\bf unbounded distributive bilattices}
\end{itemize}
 The first and third of these examples have been considered previously, ad hoc, and our  task is to bring them
within the scope of  our general theory.   For varieties
of Sugihara type, our theorems find immediate
application in \cite{CPfree}.
There is a symbiotic relationship between 
Section~\ref{sec:exs} and the theoretical
material in preceding sections. The behaviour of our examples
has motivated the theory and these are used to illustrate features
of it.  In the other direction  
our key
 theorems should be seen as  enablers.
The job  of finding 
the piggyback relations in Theorem~\ref{thm:multpig} and the partial order defined
in Section~\ref{sec:RevEng}  is specific to each application, and can be onerous. We  mention salient 
points only and refer the reader to 
 the appropriate papers for descriptions of the 
resulting dualities.  
 
We  issue one  
claimer. 
Our  Reconciliation Theorem~\ref{thm:RevEng}
describes the Priestley duals of the lattice reducts
of the algebras in the class~$\CA$ under investigation.  In this paper  we do not seek in general to upgrade the
description so as to tie together the natural
duality for~$\CA$ and a restricted Priestley
duality for~$\CA$.  See the remarks  at the end of
Section~\ref{sec:exs}.

Finally, in Section~\ref{sec:onebound}, we outline
the modifications needed to encompass classes of 
algebras with reducts in distributive lattices with
one distinguished bound.  This is appropriate since
we have an application pending.  In connection with
our on-going study of Sugihara algebras and monoids
we wish to apply our methodologies to Brouwerian
algebras  (see \cite{FG19}   for their relevance
and
 \cites{BD76,CPGo} for related material).


\section{Duality  
for unbounded distributive lattices}\label{sec:DL}


We begin by recalling  the  
duality between
 $\CCD_u$  and the category $\CP_u$ of  doubly-pointed Priestley spaces.   
The class~$\CCD_u$  may be defined to be 
the quasivariety (in fact a variety) $\ISP(\two)$,
where $\two$ is the lattice $ (\{ 0, 1\}; \lor, \land)$
in which the underlying order is given by $0 < 1$.
This class is made into a category by taking the morphisms to be all homomorphisms.  On the dual 
side, we take $\twoT$ to be the ordered topological space $(\{0,1\};   \leqslant, \boldsymbol 1,  \boldsymbol 0, \Tp)$, where $\leqslant$ is the partial order for which $0 < 1$ and $\boldsymbol 1$ and $\boldsymbol 0$
are $1$ and $0$ now regarded as nullary operations,
 and~$\Tp$ is the discrete
topology.  Then we can realise the objects of $\CP_u$ as  the class 
$\IScPn (\twoT)$, 
consisting of all isomorphic copies of
closed substructures of non-empty powers of~$\twoT$.  Morphisms are the continuous order-preserving maps which preserve the pointwise liftings
of the nullary operations. 
A self-contained account can be found in \cite{CD98}*{Chapter~1}.

\begin{thm}[Priestley duality for $\CCD_u$] \label{thm:DPu}
There exist well-defined contravariant hom-functors $\fnt{H}_u\colon \CCD_u \to \CP_u$ and $
\fnt{K}_u
\colon  \CP_u \to \CCD_u$ which set up a dual equivalence between~$\CCD_u$ and~$\CP_u$.
The functors are given as follows,  where  $\le$ is to be interpreted as  `regarded as a substructure of'.

\noindent{On objects, }
\abovedisplayskip=0pt
\belowdisplayskip=0pt
\begin{align*}
\forall \B \in \CCD_u \  \bigl( \fnt{H}_u(\B) &= \CCD_u(\B,\two)  \le  \twoT^B
  \bigr),\\
\forall \Y \in \CP_u\  \bigl( \fnt{K}_u(\Y) &= \CP_u(\Y,\twoT)  \le \two^Y\bigr); \\[-3ex]
\end{align*}
\noindent on morphisms
\abovedisplayskip=0pt
\begin{align*}
\forall f  \in \CCD_u (\B,\C) \ \bigl( \fnt{H}_u(f) &= 
- \circ f \bigr), \\
\forall \varphi   \in \CP_u (\Z,\Y) \ \bigl( \fnt{K}_u(\varphi) &= 
- \circ \varphi \bigr)\\
 \shortintertext{(here  $\B,\C \in \CCD_u$ and $\Z,\Y \in \CP_u$).}
&\\[-5ex]
\end{align*}

Moreover, a $\CP_u$-morphism $\varphi$ is surjective if and only if $\fntK_u(\varphi)$ is injective.  
\end{thm}

See for example 
 \cite{CD98}*{Subsection 1.2.5} where the result is proved directly and used as an 
appetiser for the general theory developed in subsequent chapters.  The final assertion is needed
in  the proof
of Lemma~\ref{jointsurj}.  It is a
 consequence of the duality being strong; see 
\cite{CD98}*{Chapters 3, 4}.  However it is easy to 
construct a direct proof which bypasses the notion of strongness.  

Note that the distinguished upper and lower bounds in $\fntH_u(\B)$
are given by the constant maps to~$1$ and to~$0$.


We adopt the following notation for the natural evaluation maps:  for $\B\in\CCD_u$ we write
$\ksubB$ for the map $b \mapsto \fntK_u\fntH_u(b)$
($b \in \B$).  We shall make explicit use of the fact  that these evaluation maps are surjective.


\section{Multisorted dualities and piggybacking}
 \label{sec:pigprelim}
In this section we give a summary of the rudiments
of multisorted duality theory and the idea behind the piggybacking method.  Any reader conversant  with
\cite{CD98}*{Chapter~7} or other sources describing 
piggybacking over bounded distributive lattices will find little here that is 
conceptually novel.

We shall, as needed, make use of basic facts and standard notation from universal algebra.  
Usually the classes of algebras we  consider will
be finitely generated varieties which are also
quasivarieties.  
We will be working in a setting in which J\'onsson's  Lemma applies, so that any
finitely generated variety is expressible as 
$\ISP(\CM)$ for some finite set~$\CM$ of finite algebras.

Unless indicated otherwise
we shall restrict attention to a class
$\CA$  
satisfying the following assumptions  (we shall add to this list as we proceed).
\begin{enumerate}
\item[(A1)]  There is a forgetful functor~$\fntU\colon
\CA \to \CCD_u$.

\item[(A2)]  $\CA$ is a finitely generated quasivariety
or variety expressed in the form  $\CA = \ISP(\CM)$,
where $\CM $ is a finite set of
pairwise disjoint (formally,  disjointified)
 finite algebras in $\CA$.
\end{enumerate}
 We refer to the members of $\CM$
as \defn{sorts}.  We may have distinct sorts which are isomorphic.

The definition of 
 a \emph{compatible alter ego $\CMT$ for
$\CM$} can be found in \cite{CD98}*{Section~7.1}   (or see \cite{CPmult}*{Section~3}).  The universe of the alter ego~$\CMT$ is $N: =
\dotbigcup
\{\, M \mid \M \in \CM\,\}$, the union of the universes
of the sorts.  We shall equip 
  the universe~$N$ of $\CMT$ 
 with 
the union topology~$\Tp$ obtained when  each~$M$ (for $\M \in \CM$) is  discretely topologised.
For the purposes of the theory developed in this paper, it will be sufficient to consider an alter ego which takes the form
$\CMT = (N;  G, R \cup S, \Tp)$ where 
\begin{itemize}
\item
$G \subseteq
\bigcup \{\, \CA (\M, \M') \mid \M, \M' \in \CM\,\}$
is a set of  homomorphisms between sorts;
\item  $R$ is a set of relations of arity~$2$,
 each of which is the universe of  a subalgebra of
 some $\M \times \M'$, where
$\M,\M'\in \CM$, and $S$ is a set of unary relations 
each of which is the universe of a $1$-element
subalgebra of some $\M \in \CM$.
\end{itemize}

We refer to \cite{CD98}*{Section~7.1} for a full discussion of how $\CX := 
\IScPn(\CMT)$, the 
 topological quasivariety  generated by $\CMT$, is defined.  
Here we recall only that the objects of~$\CX$ are isomorphic copies of closed substructures of powers of 
$\CMT$ (with a non-empty index set); the key feature is that powers are formed `by sorts'.  A member
$\X$ of $\CX$ is a multisorted structure of the same type as~$\CMT$, and we denote its $M$-sort by
$\X_M$. 
 Members of~$G$ and~$R\cup S$ are lifted pointwise to~$\X$. 
We shall write $r^\X$ for the lifting of a relation~$r$ to~$\X$, and similarly for elements of~$S$ and~$G$.

We then set up  hom-functors 
$\D \colon\CA \to \CX$ and $\E\colon \CX \to\CA$ 
using 
$\CM$ and its alter ego $\CMT$: 
\begin{alignat*}{2} 
                \D (\A)&= \dotbigcup \{\, \CA(\A,\M) \mid
\M 
\in \CM\,\} , \qquad 
&
                \D (f)&  = - \circ f; \\
    \E (\X)& = \CX(\X,\CMT),  &
             \E (\varphi)  &= - \circ \varphi .  
\end{alignat*}
Here the disjoint union of the hom-sets 
is  a 
(necessarily closed) substructure of 
$\dotbigcup \{\, M^A \mid \M \in \CM\,\}$,
 and so a member of $\IScPn(\CMT)$. 
 As a set, 
 $\E (\X) = \CX(\X,\CMT)$
is the collection of continuous structure-preserving maps
 $ \varphi\colon \X \to\CMT$ which are such that 
$ \varphi(\X_M )\subseteq M$ for each sort $
\M$.
This set acquires the structure of a member of 
$\CA$ by virtue of viewing it as a subalgebra of the power 
$\prod \{ \M^{\X_M} \mid \M \in \CM\,\}$.

General theory ensures that for any compatible alter ego these functors are well defined and set up a dual adjunction for which the unit and co-unit maps 
are given by evaluation maps which are embeddings.  
  We have a \emph{duality}
if for each $\A\in \CA $ the evaluation  $\esubA \colon \A \to \E\D(\A)$ is  surjective, and hence an isomorphism. In this situation we shall say that
\emph{the alter ego $\CMT$ dualises $\CA$}.  
A duality is \emph{full}, and so a dual equivalence, if the co-unit maps are also
surjections. 
(Fullness, when needed, is usually  obtained by showing the duality in question
is strong.)

We stress a  point about our objectives.
We are not seeking a duality which is strong (and hence full).  
 Strongness
may be crucial for certain applications.  This is the
case 
when natural duality methods are employed to study
admissible rules for propositional logics, as in
\cites{CFMP,CPsug},  and \cite{CPmult} also exploits
strong dualities. 
In our paper  \cite{CPfree}, on free Sugihara algebras and free Sugihara monoids,
 we do not need strong dualities, and this paper 
does not seek  such dualities.

 We are ready to head for our piggyback duality theorem.  The idea  
is to exploit 
to  maximum advantage the forgetful functor 
 $\fntU \colon \CA \to \CCD_u$ in order to align
as closely as possible the natural dual $\D(\A)$ of each $\A \in \CA$ and the Priestley dual 
$\fntH_u\fntU(\A)$
of the reduct of~$\A$. 
We first establish links between $\CA$ and its 
image under~$\fntU$. 
In particular, for each $\M \in \CM$
 we consider maps in 
$\fntH_u \fntU (\M)$, the first dual of~$\M$
under the duality in Theorem~\ref{thm:DPu}. Maps of this type are central to the piggybacking method.  
We refer to them as \emph{carrier maps}, or
\emph{carriers} for short.
We add a further assumption:
\begin{enumerate}
\item[(A3)] Associated 
to each $\M \in \CM$ is a chosen 
non-empty 
set $\Omega_{\M} \subseteq \fntH_u \fntU (\M)$
of carriers and $\Omega:= \bigcup \Omega_\M$.
\end{enumerate}
Together, (A1)--(A3) may be seen as setting out
the framework within which we shall work, with
$\CA$ and $\fntU$ as givens, and scope to impose
conditions on the
choices of $\CM$ and~$\Omega$ as we proceed;
see Remarks~\ref{rems:JS}. Unless indicated otherwise, (A1)--(A3) are henceforth assumed to hold. 

We observe that, by
changing~$\CCD_u$ to~$\CCD$ in (A1) and making the obvious consequential changes to (A2) and (A3),
we obtain the framework assumptions for piggybacking over~$\CCD$.  In other words, only
the base variety changes.


%

The piggyback method points the way to  
a choice of alter ego $\CMT$ which we hope will  dualise~$\CA$. 
Our strategy builds on 
that of \cite{CD98}*{Theorem 7.2.1}, which originated in \cite{DP87}.  
In both bounded and unbounded  cases it
depends on proving that, for each algebra~$\A\in\CA$, the natural evaluation map
$\esubA \colon \A \to \E\D(\A)$ is surjective.
    To achieve this,
one exploits the surjectivity of the corresponding 
evaluation maps for the base duality, \textit{viz.} that for $\CCD$
or $\CCD_u$, respectively. 

We now provide the key diagrams on which our piggybacking theorem will rest.  These mimic 
the corresponding diagrams for the bounded case
given for the proof of \cite{CD98}*{Theorem~7.2.1}. 
In what follows,  $\A$ is a fixed but arbitrary algebra in $\CA$.
It is a set-theoretic triviality that the diagram in
Figure~\ref{fig:pig}(a)  yields surjectivity of the 
evaluation $\esubA$ once the  diagrams
in Figure~\ref{fig:pig}(b) 
have been constructed so that the map $\Delta$ can be defined.  Here it is essential that $\fntK_u\fntH_u
\fntU(\A) \cong \fntU(\A)$ and this tacitly demands
 that the domain of $\fntK_u$ is the \emph{whole}
of the first dual $\fntH_u\fntU(\A)$.

\begin{figure} [ht]
\begin{center}
\begin{tabular}{c|c}
\vspace*{1cm}
\begin{tikzpicture} 
[auto, 
 text depth=0.25ex,
] 
\matrix[row sep= 1.2cm, column sep= .65cm]
{
\node (A)  {$\A$};
 &&  \node (EDA)  {$\E\D(\A)$};\\
&& \node (KHUA)  {$\fntK_u\fntH_u\fnt{U}(\A)$};\\
};
\draw [->] (A) to node {$e_{\A}$} (EDA);
\draw [->] (EDA) to node  
{\begin{tabular}{l}$\exists\Delta$\\[-1ex]
\tiny{[injective]}
\end{tabular}
} (KHUA);
\draw [->] (A)  to node [swap, yshift=.25cm, xshift=.5cm] 
{\begin{tabular}{r}$k_{\fnt{U}(\A)}$\\[-1ex]
\tiny{[surjective]}
\end{tabular}
} (KHUA);
\end{tikzpicture}  
\hspace*{.3cm} 
  &
\hspace*{.3cm}
\begin{tikzpicture}
[auto, 
 text depth=0.25ex,
] 
\matrix[row sep= .7cm, column sep= .6cm]
{
\node (DAM)  {$\CA(\A,\M)$};  &&  \node (M) {$M$};\\
\node (DuUA2)  {$\CCD_u (\fnt{U}(\A), \two)$}; 
&& \node (twoT)  {$\twoT$};\\
}; 
\draw [->] (DAM) to  node [swap]  {$\Phi_{\w}$} (DuUA2); 
\draw [->] (DAM) to node {$\alpha_M$} (M);
\draw [->] (M)  to node {$\w$}  (twoT);
\draw [->] (DuUA2) to node[swap] {$\Delta(\alpha)$} (twoT);
\end{tikzpicture}
\vspace*{-1.5cm}
\\&\\
 \\
 & (for fixed $\M$ and $\w$)\\[1ex]
(a) & (b)  
\end{tabular}
\end{center}
\caption{Seeking to define $\Delta$  \label{fig:pig}}
\end{figure}

Let us spell out how  Figure~\ref{fig:pig}(b) 
is required to work.
Take 
 $\alpha \in \E\D(\A)$.  That is,
$\alpha$  is a 
multisorted morphism from $\X :=\D(\A)$ to
$\CMT$.  
Thus  $\alpha$ is a sort-preserving, continuous and structure-preserving map.   So  $\alpha_M$, the
$M$-component of~$\alpha$,  maps 
the $M$-sort 
$\X_M$ 
of $\X$ into the $M$-sort 
of $\CMT$.
Figure~\ref{fig:pig}(b) represents 
a family of diagrams, one for each choice of sort $\M$ and each 
carrier $\w \in \Omega_{\M}$.
The map
$\Phi_\w := \w \circ - $ takes the $M$-sort
of $\D(\A)$, that is,  $\CA(\A, \M)$, into 
$\fntH_u \fntU(\A)$, the dual of $\A$'s reduct in
$\CCD_u$ for the base duality as given in 
Theorem~\ref{thm:DPu}. 
Each individual
diagram must  commute, and, jointly,
they must fit together to yield a well-defined
map~$\Delta$.  
For the commutativity we need, for each $x \in \CA(\A,\M)$, 
$$  
\Delta(\alpha)(\w \circ x) = \w (\alpha_{M}(x)).
$$
We shall attempt to use this as a definition, but 
being duly mindful that well-definedness has to be considered.

 Also, as noted above, 
the success of the piggyback strategy relies on 
  $\Delta(\alpha)$
 being defined on the whole of $\fntH_u\fntU(\A)$.
For this we  require
$$
\bigcup \{ \, \img \Phi_\w \mid  \M \in \CM, \
\w \in \Omega_\M \,\}=
\fntH_u\fntU(\A).
$$
This condition is known as \emph{joint surjectivity}.

In the
bounded case the corresponding requirement  is shown
to be equivalent to a separation condition involving
the sorts, the carriers and 
the endomorphisms of sorts and
the homomorphisms  between them  \cite{CD98}*{Lemma~7.2.2}.
It turns out that in the unbounded case obtaining 
conditions for joint surjectivity is more delicate.
We devote the next section to these 
issues.

\section{Joint surjectivity and separation}\label{sec:JS}

In this short section we deal with technical matters.  It might be tempting to  
think that the unbounded case would involve making 
only routine  amendments to  results for piggybacking over $\CCD$.  
This would be misguided.

We stress once again §that it will be
crucial for the proof of Theorem~\ref{thm:multpig}
that, for each $\A \in \CA$, 
the family of maps $\{\,\Phi_{\w}\mid \w\in\Omega\,\}$ is jointly surjective,   that is, for each $y\in \CCD_u (\fnt{U}(\A),\two)$ there exist $\M\in\CM$, $\w\in \Omega_{\M}$ and $x\in \CA(\A,\M)$ such that $y=
\Phi_\w (x) := \w\circ x$.
We seek 
conditions for this to hold.  These should be 
global, meaning that they should not involve  
arbitrary algebras~$\A$. Lemmas~\ref{jointsurj}
and~\ref{lem:fullJS-triv} together achieve our
objective.  The first of these is an adaptation of
\cite{CD98}*{Lemma~7.2.2}.
Our framework assumptions (A1)--(A3) from the
previous section are still in force, but for ease of
reference we include  these in the statements of our lemmas.

We introduce the following separation condition, for 
given~$\CM$ and~$\Omega$:

\begin{widenewlist}
\item[$\text{Sep}_{\CM,
\Omega}$]
 For all $\M\in \CM$, given 
$a,b\in \M$  with $a \ne b$,  then either 
there exists $\w \in \Omega_{\M}$ such that 
$\w(a) \ne \w(b)$ or there  exist
$\M' \in \CM$,  $
\w'
\in \Omega_{\M'}$ and $u\in \CA(\M , \M')$ 
    such that 
$\w' (u(a))  \ne \w' (u(b))  
$.  
\end{widenewlist}
Later we shall consider a more refined separation condition 
which parallels that used in \cite{CD98}*{Lemma~7.2.2}, but
we do not need that yet.

Let $\A \in \CA$.
Note that $\CCD_u (\fnt{U}(\A),\two) $ will always
contain the constant maps from $\fntU(\A)$ to $\two$. We denote these by 
 $\zero$ and $\one$; the domain~$\A$ will be dictated by the context.

\begin{lem}[Restricted Joint Surjectivity]
\label{jointsurj}
Assume \text{\rm (A1)}--\text{\rm (A3)}.
Then the following
conditions    
    are equivalent.  
\begin{newlist}
\item[{\rm (1)}]  $\text{\rm Sep}_{\CM, \Omega}$  
is satisfied;
\item[{\rm (2)}] for every $\A\in\CA$ and  every $a, b \in \A$ with $a \ne b$ there exist $\M \in \CM$, $\w \in \Omega_{\M}$ and $x \in \CA(\A,\M)$ such that 
$\w(x(a)) \ne \w(x(b))$;
\item[{\rm (3)}]   for every $\A\in\CA$ and each $y \in \fnt H_u\fnt U(\A)
\setminus
\{\zero, \one
\}$ 
there exist
$\M \in \CM$, $\w \in \Omega_{\M}$ and $x \in  \CA(\A,\M)$ such that
$y = \w \circ x$.
\end{newlist} 
\end{lem}

\begin{proof}  For a given~$\A$,
consider 
	$Z:= \dotbigcup \{\, \img \Phi_\w \mid \w \in \Omega\,\}\cup\{
\zero, \one
\}$.  Since~$\Omega$ is finite,
$Z $ is a closed substructure of $ \fntH_u\fnt{U}(\A)$.  Let $\mu \colon Z \to \fntH_u\fnt{U}(\A)$
be the inclusion map. Then
(3) is satisfied  if and only if $\mu$ is surjective. 
The final  statement in Theorem~\ref{thm:DPu}
 implies that this holds 
 if and only if  $\fnt{K}_u(\mu)$ is injective,  which in turn holds   if and only if  $\fnt{K}_u(\mu)\circ {\medsub k {\kern-0.75pt\fnt{U}(\A)\kern-0.75pt}}  \colon  \fnt{U}(\A ) \to  
\fntK_u (Z)$  is 
injective.  This is exactly what~(2) asserts.  
 We have proved the equivalence of (2) and~(3).

Now assume~(1).  Since $\CA = \ISP(\CM)$, 
the homomorphisms from any given $\A 
\in \CA$ into the members of 
$\CM$ separate the points of~$\A$. It follows that for any $a \ne b$ in $\A$ there exist $\M \in \CM$ and $x \in \CA(\A,
\M)$ such that  $x(a) \ne x(b)$.  By~(1), either 
there exists $\w \in \Omega_\M$ such that 
$\w(x(a)) \ne \w (x(b))$ or there exists 
$\M'\in\CM$, $\w' \in \Omega_{\M'}$  and a homomorphism  $u \in \CA(\M,\M')$ for which 
 $\w' (u(x(a))) \ne 
\w'( u(x(b)))$.  This implies that~(2) holds.  
The converse is easy to check.
\end{proof}


We now 
present conditions which
ensure that, for any $\A \in \CA$, each of the `missing' constant maps can be 
represented in the form $\Phi_\w (x) = \w \circ x$.

\begin{lem}  \label{lem:fullJS-triv}
Assume \text{\rm (A1)}--\text{\rm (A3)}. 
Then the following are equivalent:
\begin{newlist}
\item[{\rm (1)}]  
there exist a sort 
$\M_1\in\CM$, an element $d_1\in \M_1$, and a carrier  $\w \in \Omega_{\M_1}$ such that $\{d_1\}$ is a subalgebra of $\M_1$ and $\w(d_1)=1$;
\item[{\rm (2)}]   for every $\A\in\CA$  there exist
$\M\in \CM$, $\w \in \Omega_{\M}$ and $x \in  \CA(\A,\M)$ such that
$ \w \circ x=\one \in \fnt H_u\fnt U(\A)$.
 \end{newlist}

A corresponding 
 statement can be obtained for the map $\zero$ by replacing $1$ by $0$ in {\rm (1)} and {\rm (2)} above.
\end{lem}

\begin{proof}  Since every algebra $\A\in \CA$ has a unique map into any $1$-element algebra in $\CA$, it is straightforward that (1) implies (2).

For the converse,  consider a trivial ($1$-element)
algebra $\B$ in $\CA$.  By (2), there exist $\M_1\in \CM$, $\w \in \Omega_{\M_1}$ and $x \in  \CA(\B,\M_1)$ such that
$\w \circ x=\one\in  \fnt H_u\fnt U(\B)$. Hence the image of $x$ is a $1$-element subalgebra $\{d_1\}$ of $\M_1$ and 
$\w(x(\B))=\w(\{d_1\}  )=\{1\}=\one(\B)$, 
which concludes the proof.
\end{proof}

\begin{rems}[Joint surjectivity and ways to achieve it]
\label{rems:JS}
{\rm 
We assume that $\CA$ is given and (A1) holds.
We discuss how we might vary $\CM$ and 
$\Omega$, at the same time ensuring 
 that 
joint  surjectivity holds.
We have seen that restricted joint surjectivity is equivalent  to the separation condition 
$\text{Sep}_{\CM, \Omega}$ and that joint surjectivity
holds if in addition there exist $1$-element subalgebras as demanded in Lemma~\ref{lem:fullJS-triv}.

When the piggyback method was first devised the 
aim was to obtain workable dualities in circumstances
where, for example, use of the NU Duality Theorem, 
aided by hand calculations to streamline the resulting
alter ego, was  not feasible.  The quest for simple alter egos can be seen as part of the
philosophy behind piggybacking  and this has 
influenced the formulations of piggyback theorems in the past.  This accounts in particular
for the set $G$ of operations in an alter ego  often being reduced as much as possible.

Our perspective  
is a little different.  We are interested in choosing a dualising alter ego which
makes Theorem~\ref{thm:RevEng} work as transparently as possible.  We shall postpone discussion of how this relates to  the admissible choices of $\CMT$;  see Remarks~\ref{rems:newCMT} and the examples in Section~\ref{sec:exs}.
But 
it is appropriate already here to
 consider  separation alongside options for
varying $\CM$ and $\Omega$ which   explicitly take
account of  homomorphisms between sorts. 
For a given 
subset
$G$ of
$\bigcup \{\, \CA(\M,\M')\mid \M, \M' \in \CM \,\}$ 
 we
consider 
 the condition:

\begin{widenewlist}
\item[$\text{Sep}_{\CM,G,
\Omega}$]
 For all $\M\in \CM$, given 
$a,b\in \M$  with $a \ne b$,  then either 
there exists $\w \in \Omega_{\M}$ such that
$\w(a) \ne \w(b)$ or there  exist
$\M' \in \CM$,  $
\w'
\in \Omega_{\M'}$ and $u\in \CA(\M , \M')\cap G$ 
   such that 
$\w' (u(a))  \ne \w' (u(b))  
$.  
\end{widenewlist}
For any choice of~$G$ 
obviously  $\text{Sep}_{\CM, G, \Omega}$ implies
$\text{Sep}_{\CM, \Omega}$  and hence implies 
restricted joint surjectivity.


We may  explore  the interplay between
the sets which feature  in $\text{Sep}_{\CM, G, \Omega}$, and possible tradeoffs.
Loosely,
the larger $G$ is, the fewer
sorts and/or carriers we are likely to need. 
At the  extremes,  we may seek to 
minimise the number of sorts,
 allowing multiple carrier maps on the sorts as needed, or to
maximise the
number of sorts and minimise the number of carrier maps on each.

Taking~$\CM$ as
 given,  
we can always satisfy
$\text{Sep}_{\CM,\Omega}$  by letting
$\Omega_{\M} $ contain all non-constant
maps in $\CCD_u(\fnt U(\M), \two)$ for each $\M \in \CM$.  However our aim will often be to use as few 
sorts as possible, so that it is expedient to include
endomorphisms in the alter ego  insofar as these are available.

Of  course, the choice of $\CM$ will be constrained
by  the need for the class $\CA$ we wish to study
to he expressible as $\ISP(\CM)$, for   which it is 
necessary  and sufficient that the homomorphisms from any $\A \in \CA$ into the members of $\CM$ separate the points
of $\A$. 
Regarding the choice of~$\CM$ we 
make some technical comments to  justify 
the  assumption in (A3) that
each  sort $\M \in \CM$ has a 
\emph{non-empty} set
of carriers.  Assume that $\text{Sep}_{\CM,G,\Omega}$ holds.
Suppose
$\M \in \CM$ were  such that $\Omega_{\M} =\emptyset$. Let $\CM^*= \CM \setminus \{\M\}$
and let  $G^*$ be obtained from~$G$ by deleting all maps which have $\M$ as their domain or codomain.  Then 
$\text{Sep}_{\CM^*, G^*,\Omega}$
holds.  We claim that $\CA= \ISP(\CM^*)$.  
Take
 $a \ne b$ in~$\M$.  Then $\text{Sep}_{\CM,G,\Omega}$ implies 
there exist  $\M' \in \CM^*$,  $u \in \CA(\M,\M')$  and $\w' \in \Omega_{\M'}$ such that 
$\w'(u(a)) \ne \w'(u(b))$.  Then $u(a) \ne u(b)$. 
Hence $\M \in \ISP(\CM^*)$.  Finally,
$\CA= \ISP(\CM) =  \ISP(\CM^*)$, as claimed. We
deduce that 
 $\M$
can be deleted from $\CM$.  

We conclude these remarks with comments 
stemming from Lemma~\ref{lem:fullJS-triv}. 
Typically,
the sorts we choose to satisfy (A2), (A3) and 
$\text{Sep}_{\CM.G,\Omega}$
will provide the required 
$1$-element subalgebras and  carrier maps. 
Failing this, we
can add trivial algebras to $\CM$ and  corresponding maps  to $\Omega$; this will leave
$\ISP(\CM)$ unchanged.
It might appear from its proof that 
Lemma~\ref{lem:fullJS-triv} is concerned with handling a degenerate case.  
However we shall
see in Example~\ref{ex:sug} that  
condition~(2) in the lemma can fail on non-trivial
algebras.  The same example illustrates that adding
trivial sorts can sometimes not be avoided.
 
}
\end{rems}

To  summarise, 
we bring together, and label for future use, 
conditions which ensure joint surjectivity holds
and which will feature in our main theorems.

\begin{enumerate}
\item[(G)] The set
 $G\subseteq
\bigcup \{\, \CA(\M,\M')\mid \M, \M' \in \CM \,\}$ 
is such that 
$\text{\rm Sep}_{\CM,G, \Omega}$  
is satisfied.

\item[(S1)] There exist a sort 
$\M_1\in\CM$ which has a  
$1$-element subalgebra $\{d_1\}\subseteq \M_1$ and $\w \in \Omega_{\M_1}$ such that $\w(d_1)=1$. 

\item[(S0)] there exist a sort 
$\M_0\in\CM$ which has a  
$1$-element subalgebra $\{d_0\}\subseteq \M_0$ and $\w \in \Omega_{\M_0
}$ such that $\w(d_0)=0$.
\end{enumerate}

 The statement of Theorem~\ref{thm:multpig} 
will incorporate the set $G$
and the $1$-element subalgebras into the alter ego.
Their preservation by $\CX$-morphisms will play
a central role in the proof of the theorem.

%

%



This is an opportune point at which to contrast the approach in this paper with that we used in
\cites{CPsug,CPmult}.  It has long been known that, for quasivarieties of lattice-based 
algebras in particular, dualities can be upgraded to strong dualities  by enriching the alter
ego by adding suitable partial operations;  see \cite{CD98}*{Section~7.1} and \cite{CPmult}*{Section~2}.
In the latter paper, this resulted, \textit{inter alia}, in the $1$-element subalgebras of the sorts 
being included in the alter ego as nullary operations.  In this paper we avoid working with
operations and partial operations, except for the operations that appear in (G), and (S1) and
(S0) will reveal very clearly the role of $1$-element subalgebras  in multisorted piggybacking over 
$\CCD_u$ and in the process of reconciliation that follows from it.


\section{Implementing the piggyback strategy}\label{sec:multpig}

The hard work has now been done. We are armed with
conditions for joint surjectivity to feed in to the statement of our
multisorted piggyback theorem with base variety $\CCD_u$.   Assumptions (A1)--(A3) remain in force,
and (S1) and (S0) are also required.
 As proposed in Section~\ref{sec:pigprelim},
the alter ego 
will  take the form  $\CMT := (N; G, R\cup S, \Tp)$.
With one exception we have assembled the 
 assumptions we need to impose on~$\CMT$.
We have not yet introduced  
 the binary piggyback relations which
we shall include.
For $\w \in \Omega_{\M}$ and 
$\w' \in \Omega_{\M'}$, define
$  R_{\w,\w'}$
to be
the set of relations which are universes of  maximal subalgebras of the lattice 
$$
(\w,\w')^{-1} (\leqslant) := \{\, (a,b) 
\in \M \times \M' \mid \w(a) \leqslant \w'(b)\,\}.
$$
We do not claim that $R_{\w,\w'} $ is always non-empty or, when it is, that it contains a single
element.  For further comments, see 
Section~\ref{sec:exs}.


\begin{thm}  
[Multisorted Piggyback Duality Theorem, over $\CCD_u$]\label{thm:multpig}
 Assume 
$\CA = \ISP(\CM)$, $\fntU$, and $\Omega =
\bigcup_{\M \in \CM} \Omega_{\M}$ are as in
\text{\rm (A1)--(A3)}.  

Let the alter ego $\CMT = (N; G, R\cup S, \Tp)$
be constructed as follows.

\begin{enumerate}
\item[{\rm (U)}]
The universe $N$ of $\CMT $ is 
$\dotbigcup \{\, M \mid \M \in \CM\,\}$, 
the disjoint(ified) union of the universes of the members of~$\CM$; 

\item[{\rm (G)}] the set
 $G\subseteq
\bigcup \{\, \CA(\M,\M')\mid \M, \M' \in \CM \,\}$ 
is such that 
$\text{\rm Sep}_{\CM,G, \Omega}$  
is satisfied;

\item[{\rm (R)}]
the  set $R$  is $ \bigcup_{\w,\w'\in \Omega}  R_{\w,\w'}$, where $  R_{\w,\w'}$
is 
the set of all maximal subalgebras of 
$
(\w,\w')^{-1} (\leqslant)$;

\item[{\rm (T)}]   $N$ is 
equipped with 
the union topology~$\Tp$ obtained when  each~$M$ (for $\M \in \CM$) is  discretely topologised;
\end{enumerate}
and, assuming also that \text{\rm (S1)}  and
\text{\rm (S0)}  hold,
\begin{enumerate}
\item[{\rm (S)}]  the members of  $S$ are the universes  of  $1$-element subalgebras $\{d_1\}
$ and $\{d_0\}$ of sorts
$\M_1$ and $\M_0$, respectively.
\end{enumerate}
Then $\CMT$ yields a duality on $\CA$.
\end{thm}

\begin{proof}  We refer to our
 discussion in Section~\ref{sec:pigprelim} and in particular to Figure~\ref{fig:pig}. 

Fix $\A \in \CA$ and
let $\alpha \colon  \D(\A)\to \CMT$ be an
$\CX$-morphism.
We must now  show that putting
$
\Delta(\alpha)(\w \circ x) := \w(\alpha_M(x)),
$
whenever $x \in \X_M$ and $\w \in \Omega_\M$,
gives a  \emph{well-defined} map.  
For this we can proceed as in the bounded case.  
The order $\leqslant$ on $\fntH_u \fntU (\A)$ is the pointwise lifting of the partial 
order in~$\two$ and so $=$ is $\leqslant  \, \cap\, \geqslant$.  
For each $\M,\M'\in\CM$, each 
$x\in \X_{M}$ and $x'\in \X_{M'
}$,
 and each $\w\in\Omega_{\M}$ and $\w'\in \Omega_{\M'}$,
\begin{align*} 
\Phi_\w (x) \leqslant \Phi_{\w'}(x')
 & \Longleftrightarrow \w\circ x\leqslant \w'\circ x' \text{ in } \fntH_u\fnt U(\A)\\
&\Longleftrightarrow   (\w\circ x)(a)\leqslant (\w'\circ 
x')(a) \text{ for each } a\in {\A}\\ 
& \Longleftrightarrow  \{\,(x(a),x'(a))\mid a\in 
\A\,\}\subseteq (\w,\w')^{-1}(\leqslant)\\ 
& \Longleftrightarrow  \{\,(x(a),x'(a))\mid a\in \A\, \}\subseteq r\text{ for some } r\in R_{\w,\w'}\\
& \Longleftrightarrow  (x,x') \in r^\X\text{ for some } r\in R_{\w,\w'}\\
&\Longrightarrow (\alpha_{\M}(x) , \alpha_{\M'}(x') )\in r\text{ for some } r\in R_{\w,\w'}\\
&  \Longrightarrow \w(\alpha_{\M}(x)) \leqslant
\w'( \alpha_{\M'}(x')). 
\end{align*}
This yields well-definedness and also proves that
$\Delta(\alpha )$ is order-preserving.
Joint surjectivity of
the maps $\Phi_\w$, for $\w \in \Omega$, has been engineered (by (G),
(S1) and (S0)).
It  implies that $\Delta(\alpha)$
is defined on the whole of $\fntH_u\fntU (\A)$.

 We want $\Delta(\alpha)$ to be a $\CP_u$-morphism.  
We  claim that 
$\Delta(\alpha)(\one)= 1$.
We can realise 
$
\one$ as $\w_1 \circ x_1$,
where $\M_1$ has the $1$-element subalgebra
$\{d_1\}$ and there exists $x_1 \in \CA(\A, \M_1)$
which is the constant map with image $\{d_1\}$.
By~(S), $\{d_1\}^\X $
  is preserved by 
$\alpha_{M_1}$.
  Now
$\w_1 \circ x_1 =
\one$.  Moreover
$$
\Delta(\alpha)(\one)(a)=\w_1 (\alpha_{M_1} (x_1)))(a) = \w_1(
\alpha_{M_1} (x_1(a)))=
\w_1(d_1) = 1,$$
 for all  $a \in \A.$
A similar argument proves that $\Delta(\alpha)(\zero)=0$.

Finally we need 
$\Delta (\alpha) $ to be continuous.   It suffices to prove that  $\Delta (\alpha)^{-1}(V)$ is closed whenever 
$V$ is closed in $\{ 0,1\}$.  
This is proved as in the bounded case \cite{CD98}*{Theorem~7.2.1, proof of item~(5)}. 
 Finally we  require $\Delta$ to be injective. 

Suppose $\alpha $ and $\beta$ are morphisms from $\D(\A) $ to $\CMT$ with $\alpha \ne \beta$.  Then there exists $\M \in \CM$ such that $\alpha_{M} \ne \beta_{M}$ and 
this implies there exists $x \in \D(\A)$ for which  $\alpha_{M}(x) \ne \beta_{M} (x)$.   
By (G) and the definition of $\text{\rm Sep}_{\CM,G, \Omega}$,
either there
exists $\w\in  \Omega_{\M}$ such that $\w (\alpha_{M}(x)) \ne \w (\beta_{M} (x))$ or there 
exists $\M' \in \CM$, $\w' \in \Omega_{\M'} $ and $u \in  \CA(\M,\M')\cap  G$ such that
$\w'(u ( \alpha_{M}(x))) \ne \w'(u ( \beta_{M}(x)))$.  Since $\alpha$ and $\beta$ preserve~$u$,
the latter implies $\w'( \alpha_{M}(u(x))) \ne \w'( \beta_{M}(u((x)))$. 
Hence $\Delta(\alpha) (\w \circ x) \ne \Delta(\beta) (\w \circ x)$ or 
 $\Delta(\alpha) (\w \circ u(x)) \ne \Delta(\beta) (\w \circ  u(x))$.  Either way,
$\Delta(\alpha)  
\ne \Delta (\beta)$.
\end{proof}

\begin{rems}[Varying the alter ego]
 \label{rems:newCMT}
{\rm
Now we have Theorem~\ref{thm:multpig} in place, we
comment briefly on minor variants of it.
 We may sometimes wish to use an alter ego
 which is not the `standard'
 one, but  for which the piggyback
strategy still goes through.
(Such  comments are equally relevant to the bounded setting, but have usually been made for individual
varieties as these have arisen.)

With $\CM$ and $\Omega$ fixed,  there may be scope to change the relational structure $(N; G,R\cup S)$.  In particular, $G \cup R\cup S$ may be enlarged
without destroying the duality.
We  included in $R$
the  \emph{maximal} subalgebras of sublattices
  $(\w,\w')^{-1}(\leqslant)$ and these suffice in the
proof of the duality theorem  because
the sorts are finite, 
so 
any subalgebra 
of this type  will be contained 
in some maximal one.
  As always with piggybacking, we could equally well have  put \emph{all}  subalgebras of $(\w,\w')^{-1}(\leqslant)$ into $\CMT$.  
 In the other direction,  
entailment techniques 
may allow redundant relations to be deleted from
$G \cup R \cup S$
\cite{CD98}*{Chapters~2 and~8}.  
}
\end{rems}

 \section{Reconciliation achieved}
\label{sec:RevEng}

In this section we shall show how to construct 
the Priestley dual  of the $\CCD_u$-reduct of each
algebra in a quasivariety $\CA=\ISP(\CM)$,  under
the same assumptions as  
in Theorem~\ref{thm:multpig}.  In particular we carry
forward our framework assumptions (A1)--(A3),  and 
also
(S1) and (S0), concerning  the existence of sorts having $1$-element subalgebras.  

Topology has played no active part in our arguments
so far.  In this section it comes to 
the fore, in the
proof of Theorem~\ref{thm:RevEng}.  
By contrast, the order-theoretic
proofs in this section are  elementary.
If one is interested only in studying finite algebras, a finite-level duality should suffice and topology can be suppressed. Our paper \cite{CPfree} illustrates the point.

We  first  set up some  additional notation.  This echoes  that in \cite{CPcop}*{Section~2}.
For a fixed algebra $\A \in \ISP(\CM)$  and ${\X}=\fnt D(\A)$, we form an ancillary structure
as follows. 
Let
$$  \textstyle  Y=\bigcup\{\, \X_{M}\times\Omega_{\M}\mid \M\in\CM\,\}.
$$
We equip $Y$ with 
 the binary relation  $\pc\, \,\subseteq Y^2$ defined by 
$$
    (x,\w)\pc (x',\w')\mbox{ if there exists } 
r\in R_{\w,\w'} \text{ such that } (x,x') \in r^{\X}.
$$
We shall see that $\pc$ is a pre-order.  We denote the equivalence relation
$\pc\, \cap \, \succcurlyeq $ by $\approx$ and
denote the equivalence class of $y \in Y$ by 
$[y]$.  
Assuming that $\pc$ is indeed a pre-order,  we obtain a well-defined quotient partial order
$\pcq$ on  $Y/{\approx}$ given, equivalently, by
\begin{align*}
[y] \pcq [y'] &\Longleftrightarrow u \pc u'  \text{ for all }  u \approx y \text{ and }  u' \approx y',\\
[y] \pcq [y'] &\Longleftrightarrow y \pc y'.
\end{align*}

We equip $Y$  
with the topology $\Tp_{Y}$ having as 
a base of open sets
$$
    \{\,U\times\{\w\}\mid  U\mbox{ open in } \X_{M}\mbox{ and }\w\in\Omega_{\M}\,\}.
$$
We denote by $\Tp$ 
the quotient topology on $Y/
{\approx}$ derived from
$\Tp_Y$.

   
We now let $Z:= Y/{\approx} $ and consider the quotient structure  $(Z; \pcq, \Tp)$.

Proposition~\ref{prop:Y} and
Theorem~\ref{thm:RevEng}  have the same 
assumptions as 
Theorem~\ref{thm:multpig} and employ the dualising
alter ego 
$\CMT$ specified there.  The results and
their  proofs are adaptations of the statement and proof of
\cite{CPcop}*{Theorem~{\rm 2.3}}.

\begin{prop}\label{prop:Y}  Assume that 
$\CA$ is as in 
 Theorem~{\rm\ref{thm:multpig}}  and that
$\CMT$ is  the dualising alter ego given there.
Fix $
\A \in\CA$. Let
$(Y;\pc, \Tp_Y)$ be  as defined above.
Then 
 $(\, Y;\Tp_Y)$ is compact 
and 
   the binary relation $\pc$
  is a pre-order.
\end{prop}

\begin{proof}  
The topology of $Y$
 coincides with that of the finite disjoint union of the product  spaces 
$\X_{M}\times\Omega_{\M}$,
 where $\X_{M}$ 
carries  the induced topology from $\X$ and $\Omega_{\M}$ the discrete topology. Then we use the fact 
 that $\Omega$ is finite.

Now we consider the order structure. 
Here we  need only to reinterpret  a piece of the proof of 
Theorem~\ref{thm:multpig}: the  order-preservation of the maps $\Delta(\alpha)$ can be 
seen as characterising~$\pc$.
Let $(x,\w)$ and $(x',\w')$ belong to~$Y$, and assume they are associated with sorts $\M$ and $\M'$, respectively.
Then
\begin{align*}
\w\circ x\leqslant \w'\circ x' \text{ in } \fntH_u\fnt U(\A) 
& \Longleftrightarrow  (x,x') \in r^\X\text{ for some } r\in R_{\w,\w'}\\
&
\Longleftrightarrow (x,\w) \pc  (x',\w') \mbox{ in } Y.
\end{align*}
It is straightforward to check that $\pc$ is reflexive and transitive, and therefore a pre-order. 
\end{proof}

\begin{thm}[Reconciliation Theorem]
\label{thm:RevEng}
Assume that 
$\CA$ is as in 
 Theorem~{\rm\ref{thm:multpig}} 
and that 
$\CMT$ is the dualising alter ego given there.
 Fix $
\A \in\CA$.  Let  $(Z; \pcq, \Tp)$ be the quotient
structure obtained from $(Y; \pc, \Tp_Y)$.

Define 
$\Psi \colon Z\to \fntH_u  \fnt U(\A) $ 
 by 
$
[(
x,\w)]\mapsto \Phi_{\w}(x)=\w\circ x$.

\begin{enumerate}
\item[{\rm (i)}]  $\Psi$
is well-defined and  establishes an 
order-homeomorphism between  $Z$ and $\fntH_u  \fnt U(\A) $, regarded as ordered topological spaces.
  \item[{\rm (ii)}]  Take $x_1$ and $\w_1$ such that
$\w_1 \circ x_1= \one$ and
$x_0$ and $\w_0$ such that $\w_0 \circ x_0 = \zero$.
Then the  quotient structure  
$(Z; \pcq, \Tp)$ enriched   with $z_1=[(x_1,\w_1)]$ and $z_0=[(x_0,\w_0)]$ 
is a doubly-pointed
Priestley space.
  Moreover, $\Psi$  is then  a $\CP_u$-isomorphism. 
\end{enumerate}
\end{thm}

\begin{proof}  Consider (i).  
From the proof of Proposition~\ref{prop:Y} we see that
$\Psi \colon Z\to \fntH_u \fnt U(\A) $ 
is well-defined and is  an order embedding.
 By joint surjectivity it is an order isomorphism and  consequently a bijection.  

It remains to show that $\Psi$ is a homeomorphism.  Since
$Z$ is compact and $\fntH_u\fntU(\A)$ is Hausdorff, 
it suffices to show that $\Psi$ is continuous.  
By definition of the quotient topology,  $\Psi$
 is continuous if and only if 
the map  $(x, \w) \mapsto \w \circ x$  from $(Y;\Tp_Y)$ to $\fntH_u\fntU(\A)$ is continuous.
We  now prove this.

An introductory comment may be helpful.  We shall
be working with  maps which are elements of 
spaces of 
functions from one algebra to another.   Such function spaces inherit their topology from the
power in which they sit and the topology is determined  by the topology put on the universe of 
the base.  This universe will always be either  $M$ (for some sort $\M$) or $\{ 0,1\}$, and carry the
discrete topology.  It follows from the definition
of product topology that the maps we consider
are necessarily continuous.
In particular
let $x \in \X_M = \CA(\A, \M)$ and $\w \in \Omega_\M$.
Then  
  $x$
and~$\w$ are continuous, and hence $\w \circ x \colon
\A \to \twoT$ is continuous.

By definition of the topology in $\fntH_u\fntU(\A)$, we need to prove that for each $a\in \A$ and $\delta\in\{0,1\}
 $, the set $\{\, (x,\w)\mid \w(x(a))=\delta\, \}$ is open.
  For each $a \in \A$, and each $\M\in \CM$, 
let $\pi^{\M}_a \colon u \mapsto u(a)$ be the 
$a^{\text{th}}$ coordinate
projection from  $\M^A \to \M$. Then $\pi^\M_a$ is continuous. 
For  $x\in \X_{M}$ and $\delta \in \{ 0,1\}$, 
\begin{align*}
\w(x(a))=(\w \circ x ) (a) = \delta &\Longleftrightarrow
x(a) \in \w^{-1}(\delta)\\
&\Longleftrightarrow
\pi^\M_a(x) \in  \w^{-1}(\delta)\\
&\Longleftrightarrow
x  \in (\pi^{\M}_a)^{-1}(w^{-1}(\delta)).
\end{align*}
Finally,
\begin{align*}
\{\,(x,\w)\mid \ & \w(x(a))=\delta\,\}\\
&={\textstyle\bigcup_{\M \in \CM}}
\{ \, (x,\w) \mid  x \in \X_M, \ \w \in \Omega_{\M},  \
(\w \circ x)(a) = \delta \, \}\\
&=
{\textstyle\bigcup_{\M \in \CM}}
\{ \, (x,\w) \mid   \w \in \Omega_{\M},  \
x  \in (\pi^\M_a)^{-1}(w^{-1}(\delta))\, \} \\
&={\textstyle\bigcup_{\M \in \CM}} \
{\textstyle\bigcup_{\w \in \Omega_\M}}\
(\pi^\M_a)^{-1}(w^{-1}(\delta))\times\{\w\}.
\end{align*}
This set
is open because each $(\pi^\M_a)^{-1}(w^{-1}(\delta))$ is open. 
This completes the proof of~(i).

To  prove (ii), observe that
Theorem~\ref{thm:DPu} implies that
$\fntH_u \fntU(\A)$ is a doubly-pointed 
Priestley space, with $\one$ and~$\zero$
as its top and bottom  elements, respectively.
By (i),   $\Psi\colon  Z \to \fntH_u \fntU(\A)$ is an order-homeomorphism.  Therefore
$(Z; \pcq, \Tp)$ is a Priestley space and must
possess universal bounds for its order. The
statement of~(ii) identifies  these bounds:
$(Z; \pcq, z_1, z_0, \Tp) \in \CP_u$.
(Note how (S1) and (S0)  have been brought into play.)
\end{proof}

\begin{rems}  \label{rems:postRevEng}
{\rm 

In  applications of Theorem~\ref{thm:RevEng} we are interested in identifying the pre-order~$\pc$ on~$Y$ and the associated partial order~$\pcq$ on~$Z$.  Finding  $\pc$ from the members of the sets $R_{\w,w'}$ is not made more complicated
when these sets are not singletons (see Remarks~\ref{rems:exs} for comments on when this does and does not happen).  Indeed it may
 be an advantage to work with \emph{all} piggyback relations, not just maximal ones (recall Remarks~\ref{rems:newCMT}), and deliberately to look at small, and hence simple, relations. Then  $\pc$ may be pieced together from the information so obtained by taking the transitive closure.

With the insight we get from Theorem~\ref{thm:RevEng} we can make some 
further comments on reconciliation.  The theorem 
obtains $\fntH_u \fntU(\A)$ as a quotient of the structure~$Y$, in which the roles of the sorts and the
carriers, and the sorts of the natural first dual, can be 
clearly seen.  

As an example,
consider 
a situation in which
we start from a quasivariety $\ISP(\M)$ where 
$\M$ has a reduct in $\CCD_u$.  Take
$\Omega := \fntH_u \fntU(\M)$ and  one copy 
$\M_\w$ of
$\M$ for each non-constant $\w\in \Omega$ and let $\CM = \{ \M_\w\}_{\w \in \Omega}$. Then
$\text{Sep}_{\CM, \Omega}$
 is
guaranteed  to hold. 
  This may be viewed
as  a brute-force approach, with the likelihood of 
much collapsing at the quotienting stage.
(It could be necessary also to include $1$-element
sorts to allow for the constant maps and satisfy
(S1) and (S0).)
See Example~\ref{ex:sug} for an illustration.

There are circumstances in which it may be advantageous  to approach reconciliation from
both directions. 
If, as might be the case, we already
have a restricted Priestley duality for a class~$\CA$, we will
know what the quotient must look like. This may assist
us  
in optimising the choice of sorts for a natural duality.  
See the examples in Section~\ref{sec:exs} for illustrations.  

}
\end{rems}

\section{Examples} \label{sec:exs}

Here we show how our  results in 
Sections~\ref{sec:multpig} and~\ref{sec:RevEng} apply to
various classes of algebras.

\begin{ex}[Kleene lattices]  \label{ex:Klat}
{\rm 
The variety $\Klat$ of Kleene lattices is the 
unbounded analogue of the variety $\Kalg$ of Kleene algebras.  The latter class has been exhaustively studied
within natural duality theory since its inception
and provided the original motivation for the introduction
of multisorted dualities.  Kleene lattices have attracted 
less attention, but have recently come to prominence 
through the study of models for many-valued logics and in
particular the development of the theory of varieties of 
Sugihara type,  whose algebras have reducts in $\Klat$.  
In \cite{FG19}, Fussner and Galatos establish a single-sorted
strong duality for $\Klat$.  In \cite{CPmult}*{Section~6}
dualities for $\Kalg$ were summarised, and a two-sorted
strong duality for $\Klat$ was outlined, but without proof.
Here we provide a justification for a duality based on
Theorem~\ref{thm:multpig}.

 We note  that $\Klat = \ISP(\boldsymbol 3) = \HSP(\boldsymbol 3)$, where $\boldsymbol 3$ is the three-element chain in $\CCD_u$ with universe $\{0,a,1\}$ and $0 < a< 1$
equipped with negation  $\neg$ given by $\neg 0=1$,
$\neg a = a$ and $\neg 1 = 0$.  We treat 
$\Klat$ as $\ISP(\CM)$, where
$\CM= 
\{ \boldsymbol 3^-, \boldsymbol 3^+\}$ and each sort is a copy of $\boldsymbol 3$.  We use a single non-constant carrier map for 
each sort:  $a$ is sent to $1$ by $\alpha^-$ and  to~$0$   by
$\alpha^+$.  
The separation condition  
$\text{Sep}_{\CM, \Omega}$ is satisfied.
Each sort has a $1$-element subalgebra,  $\{a\}$,
and hence (S1) and (S0) are satisfied.

The dualising alter ego for $\Klat$ supplied 
by Theorem~\ref{thm:multpig}  contains the $1$-element  subalgebra
in each sort as a unary relation.  
The  set $G$ can be taken
to contain the identity maps from  $\three^-$ to
$\three^+$ and from $\three^+$ to  $\three^-$.
The alter ego  also has four piggyback
relations which arise as maximal subalgebras of
$(\w,\w')^{-1}(\leqslant)$, one for each of the possible choices of $\w$ and $\w'$; these  are as in the bounded case.
Theorem~\ref{thm:RevEng} now applies.  
The translation
from this natural duality to a Priestley-style duality for
$\fntU(\Klat)$
operates in the way we would expect, and  is illustrated in \cite{CPmult}*{Section~6}.  

Here we have an example, akin to that for Kleene algebras, in which there is a very tight relationship
between our $2$-sorted natural duality and the
easy adaptation to the unbounded case of traditional
Cornish--Fowler duality for Kleene algebras, 
whereby  $\Klat$ is dually equivalent to doubly-pointed Kleene spaces.  
Having a restricted Priestley duality already to hand  is 
valuable in two ways.  First of all, we know what the
first duals of the lattice reducts of our algebras look
like, order-theoretically.  This guides us to favour
a $2$-sorted duality over an equivalent   piggyback
duality with one sort and two carriers.  	In addition
we can upgrade the quotient structures supplied
by  Theorem~\ref{thm:RevEng}  to doubly-pointed 
Kleene spaces, thereby obtaining  dual  representations for the members of $\Klat$ and
not just those in $\fntU (\Klat)$.  For a full 
account of the corresponding results for Kleene
algebras, see \cite{DP87}*{Theorem~3.8}.
}
\end{ex}

\begin{ex}[Sugihara algebras and monoids]
\label{ex:sug}
{\rm  

 Building on our work in
\cite{CPsug, CPmult}  on (strong) dualities for
finitely generated quasivarieties and varieties
of Sugihara algebras and Sugihara monoids, 
we have moved on to investigate free algebras  \cite{CPfree}.
For this we need  Theorem~\ref{thm:RevEng},
and so too Theorem~\ref{thm:multpig}.  This has led
to the present paper.

 We refer to the cited papers for a
full introduction to Sugihara algebras and for proofs
of the claims we make below.   Here we recall only that we are interested in classes
$\SA_{2n+1} = \HSP(\Zed_{2n+1}) = \ISP(\Zed_{2n+1})$  (the odd case)  and also
$\HSP(\Zed_{2n}) = \ISP(\Zed_{2n}, \Zed_{2n-1})$  (the even case).  For $k$ odd or even 
the algebra $\Zed_k$ has a $\CCD_u$-reduct which is a sublattice of the chain of the lattice
of integers  $\mathbb Z$ which is equipped with its usual lattice operations.  The universe
of $\Zed_{2n+1}$ is the interval $[-n,n]:=\{\, i \in \mathbb Z \mid -n \leqslant i \leqslant n\, \}$
and that of $\Zed_{2n} $ is $[-n,n] \setminus \{0 \}$.  Each $\Zed_k$ also carries operations
$\neg$ (negation) and $\to$ (implication).  These are defined by restriction of operations on
$\mathbb Z$.  The negation is given by $a \mapsto -a$.  The formula for implication does not concern us here.

We consider dualities for
$\HSP(\Zed_{2n+1})$, 
for each $n \geqslant 1$,   and 
$\HSP(\Zed_{2n})$ for each $n \geqslant 2$.  
(We have omitted the trivial variety in 
 the odd case and a variety 
term-equivalent to Boolean algebras in the even case.)  
We wish to apply Theorem~\ref{thm:multpig}, with
our focus on ways in which  this differs from 
\cite{CPmult}*{Theorem~2.1},  whence we obtained
strong dualities in \cite{CPmult}*{Section~4}. This means that we concentrate
on the  
refined version of joint surjectivity given in Section~\ref{sec:JS}.

For $\HSP(\Zed_{2n+1})=\ISP(\Zed_{2n}
)$, we employ two sorts, denoted $\P^-$ and $\P^+$, which are
disjoint 
copies of $\Zed_{2n+1}$. Each has a single carrier map, where, 
            respectively, $\alpha^-(a) = 1 $ if and only if $a \geqslant 0$ and 
 $\alpha^+(a) = 1 $ if and only if $a \geqslant 1$.

For $\A \in \CA = \ISP(\Zed_{2n+1})$, the maps from $\A$ into the sorts
separate the points of~$\A$. 
Observe that $\{0\}$ is a $1$-element subalgebra of both $\P^-$ and $\P^+$ and that  $\alpha^-(0)=1$ and $\alpha^+(0)=0$. Hence  (S1) and (S0) hold.

We now consider the even case.  We first review what
happens for $\ISP(\Zed_{2n})$.  
In \cite{CPsug}*{Theorem~6.4}
 we presented a single-sorted 
 strong duality for this quasivariety, 
 with 
all the homomorphisms from $\fntU(\Zed_
{2n})$ as carriers.  With a small tweak this can be recast as a multisorted
duality (not claimed to be strong) which comes within
the scope of~Theorem~\ref{thm:multpig}.
We let $\CM$ contain
$2n$ sorts, each with a single carrier.
Two of the sorts are $1$-element algebras and the
remainder are copies of $\Zed_{2n}$, each with a 
different non-constant $\w$ as carrier. 
This ensures, in brute-force fashion,
that restricted joint surjectivity holds.   
Our $1$-element sorts artificially engineer that
the conditions of Lemma~\ref{lem:fullJS-triv} are met.
The proof of the lemma relied  on a trivial  algebra to witness the failure of joint surjectivity when condition~(2) in that lemma is not satisfied, even when $\text{Sep}_{\CM, \Omega}$ is. 
However $\ISP(\Zed_{2n})$  gives us the opportunity to demonstrate that joint surjectivity can fail also for non-trivial algebras.  We consider 
$\Zed_{2n}$ for which the only endomorphism 
is the identity map.  
For each non-constant $\w$, necessarily
$\w \circ \id_{\Zed_{2n}}=\w\neq \one$.  
 
In general, 
reducing the sets of sorts or
carriers may
 be thwarted because members of $\CM$ are not closed
under homomorphic images,  and this is exemplified by
\cite{CPsug}*{Theorem~6.4}. We may however be able to achieve a simpler duality by considering 
$\HSP(\CM)$ rather than $\ISP(\CM)$.
If, as in \cite{CPfree}, our interest is in free algebras, we have nothing to lose and
 much  to gain from  this change of perspective.

In \cite{CPmult}*{Theorem~4.8} we set up  a
$3$-sorted duality for 
$$
\HSP(\Zed_{2n}) = \ISP(\Zed_{2n}, \Zed_{2n-1}),
$$
 with sorts $\P^-$, $\P^+$
isomorphic to $\Zed_{2n-1}$ and~$\Q$ isomorphic to~$\Zed_{2n}$.  For the sorts of odd size, the carrier 
maps are defined as in the odd case.
The sort $\Q$ has a single carrier map $\beta$, with $\beta(a) = 1$ if and only if
$a > 0$.   Then
  $\text{Sep}_{\{\P^-, \P^+, \Q\}, \Omega}$ holds.
The sort $\Q$ has no $1$-element subalgebra but
 we can exploit the
existence of $1$-element subalgebras in $\P^-$ and $\P^+$ to show
(S1) and (S0) hold.  
The duality for $\HSP(\Zed_{2n})$ leads to a more
transparent application of Theorem~\ref{thm:RevEng}
than does the duality for $\ISP(\Zed_{2n})$.


Sugihara monoids exhibit the same features as we
have noted above for Sugihara algebras 
 and we give no details here.
 We refer the reader to \cites{CPmult,FG19}
for the definitions. 
In \cite{CPmult}*{Theorem~5.2} we set up
a $2$-sorted duality for each finitely generated 
quasivariety of odd Sugihara monoids, with  a 
single carrier map  for each sort, as for Sugihara
algebras.  As in that scenario, there is a $1$-element subalgebra,  $\{ 0\}$.  We could likewise adapt
our treatment in \cite{CPmult}*{Section~5} of the even case to fit the theorems
in this paper.

}
\end{ex}


\begin{ex}[Unbounded distributive bilattices]
\label{ex:dbilat}
{\rm
Our paper \cite{DBlat}  considered dualities 
for distributive bilattices, with the emphasis on the bounded case.  We refer the reader to \cite{DBlat} for all definitions.

 We comment briefly  on the relationship
between our duality  in Section~\ref{sec:multpig} 
 and  the duality presented for
the variety  $\class
{DB}_u$ of unbounded 
distributive bilattices in \cite{DBlat}*{Theorem~3.2 and Section~5}.

  The treatment
in \cite{DBlat}, like that in \cite{CPmult} for classes of
Sugihara type, focuses on strong dualities and is based on the NU Strong Duality Theorem.  In the case of $\class{DB}_u$, therefore, the proof relies on
the single-sorted version of \cite{CPmult}*{Theorem~2.1}.

We can now see that Theorems~\ref{thm:multpig} 
and~\ref{thm:RevEng}, in single-sorted form,  apply to 
$\class{DB}_u$,
 making use of the calculations
performed in \cite{DBlat}*{Section~5} for the identification of
the piggyback relations to be used.  At the end
of \cite{DBlat}*{Section~5} only minimal comments
are made about how to transition from the natural
duality setting to the Priestley-style duality for 
the $\CCD_u$ reducts. A full justification is provided 
by Theorem~\ref{thm:RevEng}.  

In \cite{DBlat}*{Section~6} (in the bounded case) we went to some lengths
to upgrade our duality so as to obtain a fully-fledged
restricted Priestley duality for $\class{DB}$ 
tied to our quotienting construction.   Corresponding results
can be expected for $\class{DB}_u$.

Similar remarks can be made about unbounded 
distributive pre-bilattices, for which strong multisorted dualities are discussed in \cite{DBlat}*{Section~10}.

}
\end{ex}

\begin{rems}\label{rems:exs}
{\rm
We make a few concluding remarks on what our examples have revealed. 

First of all we  note that only in the Kleene lattices 
case do we get sets $R_{\w,\w'}$ of maximal 
piggyback relations with just a single element.  This 
phenomenon
should be seen as the exception rather than the norm.  It occurs when the non-constant, non-lattice
operations are endomorphisms or dual endomorphisms, and  for some mild generalisations of this.

For non-lattice operations of arity greater than~$1$, a plethora of
maximal subalgebras may exist for any given pair of carriers.   This happens in particular when a
(non-classical) implication is present.
For Sugihara
algebras and Sugihara monoids,  for example, the
members of  a set $R_{\w,\w}$ 
of 
 piggyback relations, modulo converses, are  certain
 graphs
of endomorphisms and of non-extendable partial
endomorphisms (see \cite{CPsug}*{Sections~4 and~6}
and \cite{CPmult}*{Section~5}).

It will not go unnoticed by anyone with an interest
in algebras of Sugihara type that we have not 
discussed upgrading the quotient structures we get from Theorem~\ref{thm:RevEng} by using these
structures to host a restricted Priestley duality taking account of the implication.  This is not a straightforward matter, as one may surmise from
\cite{FG19}*{Section~2}.  Capturing the Kleene negation on the other hand is a triviality.

}
\end{rems}


\section{The case of one bound}\label{sec:onebound}
We have focused our paper on piggybacking over 
$\CCD_u$, distributive lattices without bounds.  Here
we note  the corresponding results when the 
base variety $\CCD_u$
is replaced by $\CCD^1$, distributive lattices with a top element included in the language as a constant~$1$.  The results are needed as part of our
on-going 
 structural analysis of Sugihara
algebras and monoids and it is expedient to record 
them in this paper.

We state the  analogue of Theorem~\ref{thm:DPu} in
abbreviated form. 

\begin{thm}[Priestley duality for $\CCD^1$] \label{thm:DPu1}

There exists a 
dual equivalence between $\CCD^1$ and the category $\CP^1$ of upper-pointed Priestley spaces. This is 
given by hom-functors  into $(\{ 0,1 \}; \land, \lor, 1)$  and the alter ego  $(\{ 0,1\}; \leqslant, 
\boldsymbol{1},\Tp)$.  

Moreover, a
$\CP^1$-morphism is surjective if and only if its
dual is injective.
\end{thm}

Thereafter we  modify  our framework assumptions
so
 that $\CA$ now has a forgetful functor~$\fntU$  
into~$\CCD^1$ and adapt 
Sections~\ref{sec:pigprelim}--\ref{sec:RevEng} 
for the new scenario 
by
deleting all references to
 the distinguished role that was played by~$0$. In
 particular, the constant map $\zero$ does not arise.
 Assumption (S0) is omitted and in the alter ego
 condition
(S)  now takes~$S$ to contain only
$\{d_1\}$,  as supplied by (S1).  No new arguments
are needed in the proofs of the $\CCD^1$-variants
of Theorems~\ref{thm:multpig} and~\ref{thm:RevEng}.  We have already seen that
(S1) and (S0)  operate
independently and in like fashion. To align with
this section we  gave proofs for the case that
(S1) is present. 

.

By reversing the roles of $1$ and $0$ 
corresponding statements hold  when the case variety is  the class
 of distributive lattices with a lower bound
which is included in the language.


\end{document}